\documentclass[12pt]{amsart}
\usepackage{amssymb,latexsym}
\usepackage{enumerate}

\makeatletter
\@namedef{subjclassname@2010}{%
  \textup{2010} Mathematics Subject Classification}
\makeatother

\newtheorem{theorem}{Theorem}[section]
\newtheorem{lemma}[theorem]{Lemma}
\newtheorem{proposition}[theorem]{Proposition}

\newtheorem{claim}[theorem]{Claim}

\theoremstyle{definition}

\numberwithin{equation}{section}

\usepackage{amsmath}
\usepackage{amsfonts}
\usepackage{graphicx}
\usepackage{amsthm}
\usepackage[mathscr]{eucal}
\usepackage{mathrsfs}
\usepackage{verbatim}

\begin{document}


\baselineskip=17pt



\title[Lattice Points]{On lattice points in large convex bodies}
\author[J. Guo]{Jingwei Guo}

\address{Jingwei Guo \\  Department of Mathematics\\ University of Wisconsin-Madison\\ Madison,
WI 53706, USA}
\email{guo@math.wisc.edu}

\date{}

\begin{abstract}
We consider a compact convex body $\mathcal{B}$ in $\mathbb{R}^d$ $(d\geqslant 3)$ with smooth boundary and
nonzero Gaussian curvature and prove a new estimate of $P_{\mathcal{B}}(t)$, the remainder in the lattice point problem, which improves previously known best result.
\end{abstract}


\maketitle

\section{Introduction}\label{intro}

Let $\mathcal{B}$ denote a compact convex subset of
$\mathbb{R}^d$ $(d\geqslant 3)$, which contains the origin as an interior
point. Suppose that the boundary $\partial \mathcal{B}$ of
$\mathcal{B}$ is a $(d-1)$-dimensional surface of class
$C^{\infty}$ with nonzero Gaussian curvature throughout.
The remainder in the lattice point problem is defined to be
\begin{equation}
P_{\mathcal{B}}(t)=\#(t\mathcal{B}\cap\mathbb{Z}^d)-\textrm{vol}(\mathcal{B})t^d \quad \textrm{for $t\geqslant 1$}. \nonumber
\end{equation}
People are interested in finding a number $\lambda(d)$ as small as possible such that
\begin{displaymath}
P_{\mathcal{B}}(t)=O(t^{d-2+\lambda(d)}).
\end{displaymath}

It is conjectured that $\lambda(d)=0$ for $d\geqslant 5$ and $\lambda(d)=\varepsilon$ for $d=3$ and $4$ where $\varepsilon>0$ is arbitrary. For spheres this bound is sharp in dimension $d\geqslant 4$ (cf. Walfisz \cite{walfisz}) while open in dimension $3$. Bentkus and G\"otze \cite{Bentkus-Gotze} proved the conjecture for ellipsoids in dimension $d\geqslant 9$.

For general convex bodies the problem is still open. By a combination of the Poisson summation formula and (nowadays standard)
oscillatory integral estimates, Hlawka \cite{Hlawka} obtained $\lambda(d)=\frac{2}{d+1}$.

In two joint papers Kr\"{a}tzel and Nowak \cite{nowak I, nowak II} used estimates for one
and two dimensional exponential sums to improve the exponent. They obtained $\lambda(d)=\frac{3}{2d}+\varepsilon$ for $d\geqslant 7$ among other results.

M\"{u}ller \cite{mullerII} significantly sharpened their result by extending their estimate to a $d$-dimensional version and he obtained
\begin{equation}
\lambda(d)=\left\{ \begin{array}{ll}
\frac{d+4}{d^2+d+2}+\varepsilon & \textrm{for $d\geqslant 5$}\\
\frac{6}{17}+\varepsilon & \textrm{for $d=4$}\\
\frac{20}{43}+\varepsilon & \textrm{for $d=3$}  \label{muller-exponent}
\end{array},\right.
\end{equation}
where $\varepsilon>0$ is arbitrary.

We first observe that estimates of certain oscillatory integrals in M\"{u}ller's paper can be obtained by using the method of stationary phase. This observation leads to our Proposition \ref{B-process} below which recovers M\"{u}ller \cite{mullerII} Theorem 2 without the $\varepsilon$ there.  This already leads to an improvement of \eqref{muller-exponent} with $\varepsilon=0$.

If we use asymptotic expansions of those oscillatory integrals, the leading terms form new exponential sums to which we can iterate M\"{u}ller's $d$-dimensional estimate. This iteration leads to our new estimate of exponential sums in Theorem \ref{Exp-Sum} below, which is in fact the main result of this paper. As a consequence, we can obtain the following new bound of $P_{\mathcal{B}}(t)$ for every dimension $d\geqslant 3$:
\begin{theorem}\label{lattice}
If $\mathcal{B}$ satisfies the
conditions stated above, then the bound
$P_{\mathcal{B}}(t)=O(t^{d-2+\beta(d)})$ holds for
\begin{equation}
\beta(d)=\bigg\{ \begin{array}{ll}
\frac{d^2+3d+8}{d^3+d^2+5d+4} & \textrm{for $d\geqslant 4$}\\
\frac{73}{158} & \textrm{for $d=3$} \nonumber
\end{array}.
\end{equation}
The implicit constant may only depend on the body $\mathcal{B}$.
\end{theorem}

It's not hard to check our estimate is indeed sharper than \eqref{muller-exponent}. In particular, for large $d$ this is clear because $\beta(d)=\frac{1}{d}+\frac{2}{d^2}+O(\frac{1}{d^3})$ while $\lambda(d)=\frac{1}{d}+\frac{3}{d^2}+O(\frac{1}{d^3})+\varepsilon$.

For more results of the problem (e.g. average and lower bounds of the remainder) the reader could check \cite{nowak I}, \cite{nowak II}, \cite{mullerI}, \cite{mullerII}, \cite{nowak-seq-1}, \cite{nowak-seq-2}, and \cite{nowak-seq-3}.

In the case of planar domains, the sharpest known bound
\begin{displaymath}
P_{\mathcal{B}}(t)=O(t^{\frac{131}{208}}(\log t)^{2.26})
\end{displaymath}
is due to Huxley \cite{huxley2003}, who applied his refined variant of the ``Discrete Hardy-Littlewood Method'' originally due to Bombieri, Iwaniec, and Mozzochi. Huxley's method beats the classical theory of exponential sums, but it seems to be purely two dimensional. In this paper, we focus on higher dimensions and our main tools are still from the classical theory.

\medskip

{\it Notation:}  We use the usual Euclidean norm for a vector. $B(x, r)$ represents the usual Euclidean ball centered at $x$ with radius $r$. The norm of a matrix $A\in \mathbb{R}^{d\times d}$ is given by $\|A\|=\sup_{|x|=1}|Ax|$. $e(f(x))=\exp(-2\pi i f(x))$ and $\mathbb{Z}_{*}^{d}=\mathbb{Z}^{d}\setminus \{0\}$. For a set $E\subset \mathbb{R}^d$ and a positive number $a$, we define $E_{(a)}$ to be a larger set
\begin{equation}
E_{(a)}=\{x\in \mathbb{R}^d: \textrm{dist}(E, x)<a \}. \nonumber
\end{equation}

We use the differential operators
\begin{equation}
D^{\mu}_{x}=\frac{\partial^{|\mu|}}{\partial x_1^{\mu_1} \cdots \partial x_d^{\mu_d}} \quad \big(\mu=(\mu_1, \ldots, \mu_d)\in \mathbb{N}_0^d, |\mu|=\sum_{i=1}^{d}\mu_i \big)\nonumber
\end{equation}
and the gradient operator $D_x$. We often omit the subscript if no ambiguity occurs.

For functions $f$ and $g$ with $g$ taking non-negative real values, $f\lesssim g$ means $|f|\leqslant Cg$ for some constant $C$. If $f$ is also non-negative, $f\gtrsim g$ means $g\lesssim f$. The Landau notation $f=O(g)$ is equivalent to $f\lesssim g$. The notation $f\asymp g$ means that $f\lesssim g$ and $g\lesssim f$.

We will adopt a convention due to Bruna et al. \cite{bruna} and say that a constant is \textit{allowable} if it only depends on the body $\mathcal{B}$. Throughout this paper except Section \ref{prelim}, all constants implied by the notation $\lesssim$, $\gtrsim$, $\asymp$, and O($\cdot$) are allowable.

Wherever a variable occurs as a summation variable the reference
is to integral values of the variable.

\medskip

{\it Structure of the paper:} In Sect.\ref{prelim}, after several preliminary lemmas we prove three estimates of exponential sums. In particular, the last one is the main result of this paper. In Sect.\ref{nonvanishing}, we show that certain matrices have nonvanishing determinants and
their entries satisfy some size estimates. In Sect.\ref{mainproof}, we put these ingredients together to prove Theorem \ref{lattice}. At last we put one quantitative version of inverse function theorem in the appendix.

\section{Estimates of Exponential Sums}\label{prelim}

The classical theory of exponential sums has two processes: the Weyl-van der Corput inequalities (A-process) and the Poisson summation formula followed by the method of stationary phase (B-process). Before the estimation of exponential sums, we first introduce two preliminary lemmas related to these two processes.

\medskip

For integrals in the form
\begin{equation}
I(\lambda)=\int_{\mathbb{R}^d}w(x)e^{i\lambda f(x)}dx,  \nonumber
\end{equation}
H\"{o}rmander \cite{hormander} Theorem 7.7.5 gives an asymptotic formula for the case when the phase function $f$ has a nondegenerate critical point. It is one of the expressions of the method of stationary phase and we only need it when $f$ takes real value.

\begin{lemma}\label{SP}
Let $K\subset \mathbb{R}^d$ be a compact set, $X$ an
open neighborhood of $K$, and $k$ a positive integer. If $f$ is real and in $C^{\infty}(X)$, $w\in C_{0}^{\infty}(K)$, $D f(x_0)=0$,
det$(D^2 f(x_0))\neq 0$, and $D f\neq 0$ in $K\setminus \{x_0\}$, then
\begin{align}
\big|I(\lambda)-(2\pi)^{\frac{d}{2}}&e^{i(\frac{\pi}{4}\textrm{sgn}D^2f(x_0)+\lambda f(x_0))}|\textrm{det}(D^2f(x_0))|^{-\frac{1}{2}} \lambda^{-\frac{d}{2}}\sum_{j=0}^{k-1}\lambda^{-j}L_jw\big|\nonumber \\
     &\leqslant C \lambda^{-k} \sum_{|\mu|\leqslant 2k} \sup\limits_{x}|D^{\mu} w(x)|, \quad \textrm{for $\lambda>1$}.\nonumber
\end{align}
Here $C$ is bounded when $f$ stays in a bounded set in $C^{3k+1}(X)$ and $|x-x_0|/|Df(x)|$ has a uniform bound. With
\begin{equation}
g_{x_0}(x)=f(x)-f(x_0)-\langle D^2f(x_0)(x-x_0), x-x_0\rangle/2 \nonumber
\end{equation}
which vanishes of third order at $x_0$ we have
\begin{equation}
L_j w=\sum_{\upsilon-\gamma=j}\sum_{2\upsilon\geqslant 3\gamma}i^{-j}2^{-\upsilon}\langle D^2f(x_0)^{-1}D, D\rangle^{\upsilon}
(g^{\gamma}_{x_0}w)(x_0)/\upsilon!\gamma!. \nonumber
\end{equation}
\end{lemma}

{\it Remark:} 1) $L_j$ is a differential operator of order $2j$ acting on $w$ at $x_0$. The sum has only a finite number of terms for each $j$.

2) The integral $I(\lambda)$ has the following asymptotic expansion:
\begin{equation}
I(\lambda)=\lambda^{-\frac{d}{2}}\sum_{j=0}^{N_1}a_j \lambda^{-j}+O(\lambda^{-\frac{d}{2}-N_1-1})   \quad \textrm{for any $N_1\in \mathbb{N}$}.     \nonumber
\end{equation}
The constant implied in the error term depends on $d$, $N_1$, size of $K$, upper bounds of finitely many derivatives of $w$ and $f$ in the support of $w$, and lower bound of $|\textrm{det}(D^{2}f(x_0))|$. Each coefficient $a_j$ depends on $d$, $j$, values of finitely many derivatives of $w$ and $f$ at the point $x_0$, and value of $|\textrm{det}(D^{2}f(x_0))|$. These $a_j$'s have explicit formulas, in particular
\begin{equation}
a_0=(2\pi)^{\frac{d}{2}}w(x_0)e^{i(\frac{\pi}{4}\textrm{sgn}D^2f(x_0)+\lambda f(x_0))}/|\textrm{det}(D^2f(x_0))|^{\frac{1}{2}}. \nonumber
\end{equation}

\medskip

Suppose $M>1$ and $T>0$ are parameters. We consider $d$-dimensional exponential sums of the form
\begin{equation}
S=S(T, M; G, F)=\sum_{m\in\mathbb{Z}^d} G(\frac{m}{M}) e(TF(\frac{m}{M})), \label{expsum-S}
\end{equation}
where $G:\mathbb{R}^d\rightarrow\mathbb{R}$ is $C^{\infty}$ smooth, compactly supported, and bounded above by a constant, and
$F:\Omega\subset \mathbb{R}^d \rightarrow\mathbb{R}$ is $C^{\infty}$ smooth on an open convex domain $\Omega$ such that
\begin{equation}
\textrm{supp} (G)\subset \Omega \subset c_0 B(0, 1),  \label{size-of-domain}
\end{equation}
where  $c_0>0$ is a fixed constant.

We are interested in finding upper bounds of $S$ in terms of $T$ and $M$. Exponential sums of the form \eqref{expsum-S} are essentially the same as those considered in M\"{u}ller \cite{mullerII}. In lower dimension Huxley studied sums in a similar but more complicated form, for example, see \cite{huxley2003}.

The following lemma is a variant of M\"{u}ller \cite{mullerII} Lemma 1, namely the so-called iterated one-dimensional Weyl-van der Corput inequality.

\begin{lemma}\label{weyl}
Let $q\in \mathbb{N}$, $G$, $F$, and $S$ be as above, and $r_1, \ldots,
r_q\in \mathbb{Z}^d$ be nonzero integral vectors with
$|r_i|\lesssim 1$. Furthermore, let $H$ be a real parameter which
satisfies $1<H\lesssim M$. Set $Q=2^q$ and
$H_l=H_{q,l}=H^{2^{l-q}}$ for $1\leqslant l\leqslant q$. Then
\begin{equation}
|S(T, M; G, F)|^{Q}\lesssim \frac{M^{Qd}}{H} + \frac{M^{(Q-1)d}}{H_1\cdots
H_q} \sum_{\substack{1\leqslant h_i<H_i \\1 \leqslant i \leqslant
q}} |S(\mathscr{H}TM^{-q}, M; G_q, F_q)|, \nonumber
\end{equation}
where $\mathscr{H}=\prod_{l=1}^{q}h_l$ and functions $G_q$, $F_q$ are defined as follows:
\begin{equation}
G_q(x)=G_q(x, h_1, \ldots, h_q)=\prod_{\substack{u_i\in\{0,1\}\\
1\leqslant i\leqslant q}} G(x+\sum_{l=1}^{q} \frac{h_l}{M} u_l r_l)  \nonumber
\end{equation}
and
\begin{align}
&F_q(x)=F_q(x, h_1, \ldots, h_q) \nonumber \\
            &=\int_{(0,1)^q} (r_1\cdot D)\cdots(r_q\cdot D)F
       (x+\sum_{l=1}^{q} \frac{h_l}{M} u_l r_l)du_1\ldots du_q.\nonumber
\end{align}

The integral representation of $F_q$ is well defined on
the open convex set $\Omega_q=\Omega_q(h_1, \ldots,
h_q)=\{x\in \Omega: x+\sum_{l=1}^{q} (h_l/M) u_l r_l \in
\Omega \textrm{ for all } u_l\in \{0,1\}, 1\leqslant l\leqslant q\}$.
supp$(G_q)\subset \Omega_q\subset \Omega$.
\end{lemma}

We give, without a proof, an easy but useful result concerning the distance between the boundary of supp($G_q$) and $\Omega_q$.
\begin{lemma}\label{distance-boundary}
For fixed $(h_1, \ldots, h_q)$,
\begin{displaymath}
\textrm{dist}(\textrm{supp}(G), \Omega^{c}) \geqslant c_1\Rightarrow\textrm{dist}(\textrm{supp}(G_q), \Omega_q^{c}) \geqslant c_1.
\end{displaymath}
\end{lemma}

The exponential sum $S$ is bounded by $CM^d$ trivially, but we lose cancelation by just putting absolute value on each term. Below we will prove three bounds of $S$ obtained by applying various combinations of A- and B-processes. In the statement of these results we will assume derivatives of $G$ and $F$ up to certain orders are uniformly bounded. The orders may not be optimal but sufficient for the proof.

We first prove a bound of $S(T, M; G, F)$ by applying a B-process. For an analogous result in $1$-dimension, see Theorem 2.2 in \cite{expsum}.

\begin{proposition}[Estimate by a B-process]\label{B-process} Let $d\geqslant 2$. Assume that dist(supp$(G)$, $\Omega^{c}$)$\geqslant c_1$ for some constant $c_1$ and for all $x\in\Omega$ and $\nu\in\mathbb{N}_{0}^{d}$ with $|\nu|\leqslant 3\lceil\frac{d}{2} \rceil +1$
\begin{equation}
(D^{\nu}G)(x)\lesssim 1, \label{aaaa}
\end{equation}
\begin{equation}
(D^{\nu}F)(x)\lesssim 1, \label{bbbb}
\end{equation}
and
\begin{equation}
|\textrm{det} (D^2 F(x))|\gtrsim 1, \label{cccc}
\end{equation}
then
\begin{equation}
S(T, M; G, F)\lesssim T^{\frac{d}{2}}+M^{d}T^{-\frac{d}{2}}. \label{B-process-est}
\end{equation}

The implicit constant in \eqref{B-process-est} depends on $d$, $c_0$, $c_1$, and constants implied in \eqref{aaaa}, \eqref{bbbb}, and \eqref{cccc}.
\end{proposition}

\begin{proof}[Proof of Proposition \ref{B-process}]
Applying to $S$ the $d$-dimensional Poisson summation formula followed by a change of variables $y=M x$ yields
\begin{equation}
\begin{split}
S(T, M&; G, F)=\sum_{p\in\mathbb{Z}^d}\int_{\mathbb{R}^d} G(\frac{y}{M})
e(TF(\frac{y}{M})-y\cdot p)dy \\
      &=\sum_{p\in \mathbb{Z}^d} M^d \int G(x) e(TF(x)-Mx\cdot p) dx,
\end{split}\label{eee}
\end{equation}

By \eqref{bbbb}, there exists a sufficiently large constant $A_0$ such that
\begin{equation}
|D F(x)|\leqslant A_0/2.  \nonumber
\end{equation}
We split the sum in \eqref{eee} into two parts, namely
\begin{equation}
S(T, M; G, F)=\sum_{|p|\geqslant A_0 T/M}+\sum_{|p|<A_0T/M}=:\textrm{I}+\textrm{II}, \nonumber
\end{equation}
and distinguish the estimation into two cases.

{\bf Case 1.} If $|p|\geqslant A_0T/M$.

Let $\Psi(x, p)=(TF(x)-Mx\cdot p)/(M|p|)$, then
\begin{equation}
\textrm{I}=\sum_{|p|\geqslant A_0T/M} M^d \int G(x) e(M|p|\Psi(x, p)) dx.  \nonumber
\end{equation}

Under the given assumptions, for all $x\in \Omega$ and $|\nu|\leqslant d+2$ we have $D^{\nu}G(x)\lesssim 1$, $D^{\nu}_x \Psi(x,p) \lesssim 1$, and also
\begin{equation}
|D_x \Psi+
p/|p||=| T DF(x)/(M|p|)|\leqslant 1/2,    \nonumber
\end{equation}
which ensures $|D_x\Psi|\geqslant 1/2$. By integration by parts (H\"{o}rmander \cite{hormander} Theorem 7.7.1) we get
\begin{equation}
\int G(x) e(M|p|\Psi(x, p)) dx\lesssim (M|p|)^{-d-1} \nonumber
\end{equation}
which leads to
\begin{equation}
\textrm{I}\lesssim M^{-1} \sum_{p\in \mathbb{Z}_{*}^{d}} |p|^{-d-1}\lesssim M^{-1}.    \nonumber
\end{equation}

{\bf Case 2.} If $|p|<A_0T/M$.

Let $\Phi(x, p)=F(x)-(M/T) x\cdot p$, then
\begin{equation}
\textrm{II}=\sum_{|p|< A_0T/M} M^d \int
G(x) e(T\Phi(x, p)) dx.  \nonumber
\end{equation}

If $T\leqslant 1$, sum II $\lesssim M^d\leqslant M^d T^{-d/2}$.

If $T>1$, we claim that each integral in sum II is $\lesssim T^{-d/2}$. Assume this for a moment, then
\begin{equation}
\textrm{II} \lesssim (1+(T/M)^d)M^d T^{-\frac{d}{2}}=T^{\frac{d}{2}}+M^d T^{-\frac{d}{2}}. \nonumber
\end{equation}

Observe that above bound is true for sum II no matter whether $T\leqslant 1$ or $T>1$. It follows that
\begin{equation}
S(T, M; G, F)\lesssim M^{-1}+T^{\frac{d}{2}}+M^d T^{-\frac{d}{2}}\lesssim T^{\frac{d}{2}}+M^d T^{-\frac{d}{2}},\nonumber
\end{equation}
which is the desired bound for $S(T, M; G, F)$. The only thing left is to prove above claim.

\bigskip

Let's fix a $|p|<A_0T/M$. For all $x\in \Omega$ and $|\nu|\leqslant 3\lceil \frac{d}{2}\rceil+1$, given assumptions imply $D^{\nu}_{x}\Phi(x, p)\lesssim 1$ and $|\textrm{det} (D^2_{xx} \Phi(x, p))|\gtrsim 1$. We first need to study
critical points of the phase function. Denote
$f(x)=DF(x)$, $\widetilde{p}=(M/T)p$, then $D_x\Phi(x, p)=f(x)-\widetilde{p}$. The critical points
are determined by the equation
\begin{equation}
f(x)=\widetilde{p}, \quad \quad \textrm{$x \in \Omega$}. \nonumber
\end{equation}

We know that $\textrm{supp}(G)$ is strictly smaller
than $\Omega$ and the distance between their boundary is larger than constant $c_1$. Let $r_0=c_1/2$. By the Taylor's formula, there exists a uniform $r_*$($<r_0$) such that if $\widetilde{x}$ is a critical point in $(\textrm{supp}(G))_{(r_0)}$\footnote{For the meaning of this notation, check Section \ref{intro}.} then $|D_x\Phi(x, p)|\gtrsim |x-\widetilde{x}|$ for any $x\in B(\widetilde{x}, r_*)$.

Applying Lemma \ref{app1} to $f$ with above $r_0$ yields two uniform positive numbers $r_1$, $r_2$ such that $2r_1\leqslant r_*$ and for any $x\in (\textrm{supp}(G))_{(r_0)}$,
$f$ is bijective from $B(x, 2r_1)$ to an open set containing $B(f(x), 2r_2)$.

If $x_1$, $x_2$ are two different critical points in $(\textrm{supp}(G))_{(r_0)}$ (if exist), then $B(x_1, r_1)$ and $B(x_2, r_1)$ are disjoint and still contained in $\Omega$. It follows, simply by a size estimate, that the number of possible critical points in $(\textrm{supp}(G))_{(r_0)}$ is bounded by a constant.

We will only consider critical points in $(\textrm{supp}(G))_{(r_1)}$ below. Denote
\begin{equation}
S_p=\{x\in \textrm{supp}(G): |D_x\Phi(x, p)|<r_2\}. \nonumber
\end{equation}

If $S_p$ is empty, which means $|D_x\Phi|$ has
a lower bound $r_2$ on $\textrm{supp}(G)$, by integration by parts the integral is of order $O(T^{-\lceil \frac{d}{2}\rceil})$.

If $S_p$ is not empty, at least one critical point exists in $(\textrm{supp}(G))_{(r_1)}$. To see this, assume $x\in S_p$ which implies $|f(x)-\widetilde{p}|<r_2$. Note that $f$ is bijective from $B(x, r_1)$ to an open set containing $B(f(x), r_2)$, hence there exists a point $\widetilde{x}\in B(x, r_1)$ such that $f(\widetilde{x})=\widetilde{p}$. This means $\widetilde{x}$ is a critical point and $x\in B(\widetilde{x}, r_1)\subset \Omega$. As a consequence, $S_p$ is contained in the union of finitely many balls centered at critical points with radius $r_1$.

Assume $\widetilde{x}_{i}(p)$ ($1\leqslant i\leqslant
J(p)$) are all critical points in $(\textrm{supp}(G))_{(r_1)}$. Let
\begin{equation}
\chi_{i}(x)=\chi((x-\widetilde{x}_{i}(p))/r_1), \nonumber
\end{equation}
where $\chi(x)$ is a given smooth cut-off function whose value is $1$ if $|x|\leqslant 1/2$ and $0$ if $|x|>1$. Let $\chi_{0}=1-\sum_{i=1}^{J(p)}\chi_{i}$, then
\begin{equation}
\begin{split}
\int
G(x) e(T\Phi(x, p)) dx  = &\sum_{i=1}^{J(p)} \int \chi_{i}(x)G(x) e(T\Phi(x, p)) dx \\
 &+\int \chi_{0}(x) G(x) e(T\Phi(x, p)) dx.\nonumber
 \end{split}
\end{equation}

For each $1\leqslant i\leqslant
J(p)$, the integral in above summation has its domain contained in $B(\widetilde{x}_i(p), r_1)$ and it is of order $O(T^{-d/2})$ by Lemma \ref{SP}.

If $x\in \textrm{supp}(G)\setminus \cup_{i=1}^{J(p)}B(\widetilde{x}_i, r_1/2)$, there exists a uniform constant $r_3$ such that $|D_x\Phi(x, p)|\geqslant r_3$. Hence the last integral above is of order $O(T^{-\lceil \frac{d}{2}\rceil})$ by integration by parts. This finishes the proof.
\end{proof}

Now we can prove another bound of $S(T, M; G, F)$ by applying A-process $q$ times (Lemma \ref{weyl}) followed by a B-process (Proposition \ref{B-process}). For analogous results in $1$-dimension, see Theorem 2.6, 2.8, 2.9 in \cite{expsum}.

\begin{proposition}[Estimate by an $\textrm{A}^q$B-process]\label{AB-process} Let $d\geqslant 3$. Assume that dist(supp$(G)$, $\Omega^{c}$)$\geqslant c_1$ for some constant $c_1$ and for all $x\in\Omega$ and $\nu\in\mathbb{N}_{0}^{d}$ with $|\nu|\leqslant 3\lceil \frac{d}{2}\rceil+q+1$
\begin{equation}
(D^{\nu}G)(x)\lesssim 1, \label{aaa}
\end{equation}
\begin{equation}
(D^{\nu}F)(x)\lesssim 1, \label{bbb}
\end{equation}
and for some
fixed $\mu\in\mathbb{N}_{0}^{d}$ with $q=|\mu|$ and all $x\in\Omega$
\begin{equation}
\big|\textrm{det} \big( \frac{\partial^{2}D^{\mu}F}{\partial x_i \partial
x_j}(x) \big) _{1\leqslant i,j\leqslant d}\big|\gtrsim 1. \label{ccc}
\end{equation}

If $T$ is restricted to
\begin{equation}
T \geqslant M^{q-\frac{2}{d}+\frac{2}{Q}} \quad (Q=2^q), \label{restriction-A^qB}
\end{equation}
then
\begin{equation}
S(T, M; G, F)\lesssim T^{w_{d,q}} M^{d-(q+2)w_{d,q}}, \label{esti-S}
\end{equation}
where
\begin{equation}
w_{d,q}=\frac{d}{2d(Q-1)+2Q}. \nonumber
\end{equation}

The implicit constant in \eqref{esti-S} depends on $d, q, c_0, c_1,$ and constants implied in \eqref{aaa}, \eqref{bbb}, and \eqref{ccc}.
\end{proposition}

{\it Remarks:} 1) If $T\gtrsim M^{q+2}$, the trivial bound
$S\lesssim M^d$ is better than the above estimate.

2) If we take $T=\Lambda M$, we immediately obtain the bound in M\"{u}ller \cite{mullerII} Theorem 2 without $\varepsilon$. As a consequence, this improves M\"{u}ller's exponent \eqref{muller-exponent} by removing the $\varepsilon$.

\begin{proof}[Proof of Proposition \ref{AB-process}] Let $e_1=(1,0,\cdots,0), \ldots,
e_d=(0,\cdots,0,1)$ denote the standard orthonormal basis of
$\mathbb{R}^d$, then $\mu=\sum_{l=1}^{q} e_{k_l}$, where $1\leqslant
k_l \leqslant d, 1\leqslant l \leqslant q$. Assume that $1<H\leqslant c_2M$ with a small constant $c_2$ (to be determined later) and that $M>c_2^{-1}$ (otherwise the trivial bound is better than \eqref{esti-S}). By Lemma \ref{weyl} with $r_l=e_{k_l}$, the estimation is reduced to that of $S(\mathscr{H}TM^{-q}, M; G_q, F_q)$. The $G_q$, $F_q$ are as defined in that lemma, so is the domain $\Omega_q$.

Note that supp$(G_q)\subset \Omega_q\subset c_0 B(0, 1)$ and $\textrm{dist}(\textrm{supp}(G_q), \Omega_q^{c}) \geqslant c_1$ by Lemma \ref{distance-boundary}. For all $x\in \Omega_q$ and $|\nu|\leqslant 3\lceil \frac{d}{2}\rceil+1$, we have $D^{\nu} G_q(x) \lesssim 1$, $D^{\nu} F_q(x) \lesssim 1$, and $|\textrm{det} (D^2 F_q (x))|\gtrsim 1$.

The two upper bounds are easy to get. The
lower bound needs some effort. We first have
\begin{equation}
\frac{\partial^{2}F_q}{\partial x_i \partial x_j}(x)=
\int_{(0,1)^q} (\frac{\partial^2
D^{\mu}F}{\partial x_i \partial x_j})(x+\sum_{l=1}^{q} \frac{h_l}{M} u_l
r_l) du_1 \ldots du_q \nonumber
\end{equation}
\begin{equation}
=\frac{\partial^2 D^{\mu}F}{\partial x_i
\partial x_j}(x)+O(\frac{H}{M}), \nonumber
\end{equation}
thus
\begin{equation}
|\textrm{det}(D^2 F_q (x))|= \big|
\textrm{det}\big( \frac{\partial^2 D^{\mu}F}{\partial x_i
\partial x_j}(x) \big)+O(\frac{H}{M}) \big|. \nonumber
\end{equation}
If $c_2$ is sufficiently small, the lower bound \eqref{ccc} and $H\leqslant c_2M$ imply that the above determinant is $\gtrsim 1$.

Applying Proposition \ref{B-process}, we get
\begin{equation}
S(\mathscr{H}TM^{-q}, M; G_q, F_q)\lesssim (\mathscr{H}TM^{-q})^{\frac{d}{2}}+M^{d}(\mathscr{H}TM^{-q})^{-\frac{d}{2}}, \nonumber
\end{equation}

Since $H_1\cdots H_q=H^{2-2/Q}$ and
\begin{equation}
\sum_{\substack{1\leqslant h_i<H_i \\1 \leqslant i \leqslant q}} \mathscr{H}^{\alpha}\lesssim
\bigg\{ \begin{array}{ll}
(H_1\cdots H_q)^{\alpha+1} & \textrm{if $\alpha>-1$}\\
1 & \textrm{if $\alpha<-1$}  \nonumber
\end{array},
\end{equation}
we get from Lemma \ref{weyl}
\begin{align}
|S(T, M; G, F)|^{Q}\lesssim M^{Qd} &H^{-1}+M^{(Q-1-\frac{q}{2})d}T^{\frac{d}{2}}(H^{2-2/Q})^{\frac{d}{2}}\nonumber \\
                    &+M^{(Q+\frac{q}{2})d}T^{-\frac{d}{2}}H^{-2+2/Q}. \nonumber
\end{align}
Balancing the first two terms yields the optimal choice
\begin{equation}
H^{(1-1/Q)d+1}=B_1 T^{-d/2} M^{(q+2)d/2}, \nonumber
\end{equation}
where $B_1$ can be chosen sufficiently small such that assumption \eqref{restriction-A^qB} implies $H\leqslant c_2 M$.
Due to Remark 1, we can assume $T\leqslant B_2M^{q+2}$ with a sufficiently small $B_2$, which implies $1<H$. With this choice of $H$ the third terms is $\lesssim M^{(Q-1)d}H^{d-1}\lesssim M^{Qd} H^{-1}$. Hence we get
\begin{equation}
S(T, M; G, F)\lesssim  M^d H^{-1/Q} \lesssim T^{w_{d,q}} M^{d-(q+2)w_{d,q}}, \nonumber
\end{equation}
where $w_{d,q}$ is as defined in the proposition.
\end{proof}

Next we will estimate $S(T, \delta M; G, F)$ where $\delta>0$ is a parameter. In the following theorem and its proof, we will follow the convention: if we write a $\delta$ in a subscript (e.g. $\gtrsim_{\delta}$,  $\lesssim_{\delta}$, $\asymp_{\delta}$, or $O_{\delta}$), we emphasize that the implicit constant depends on $\delta$; otherwise it does not.

The proof will proceed as follows. We first apply A-process $q$ times (Lemma \ref{weyl}) followed by a B-process, while in the latter process we use Lemma \ref{SP} to get the asymptotic expansions of certain oscillatory integrals. By looking at the leading terms, we obtain some new exponential sums to which we apply a AB-process (Proposition \ref{AB-process} with $q$ there being $1$).

Before we can apply Proposition \ref{AB-process}, however, we need some preparation in the first B-process. For instance, we use partitions of unity to restrict certain domains to small balls on which certain critical point function (if exists) is smooth; we distinguish the cases when we are allowed to use Lemma \ref{SP}; we prove nonvanishing determinants needed in two B-processes. One difficulty is to prove the nonvanishing determinants for the second B-process, and this is where we need the auxiliary condition \eqref{entry-estimate} below. In next section, we will show such condition is indeed satisfied in the lattice point problem.

After all these are settled, the $\textrm{A}^q$BAB-process finally leads to the following theorem:

\begin{theorem}[Estimate by an $\textrm{A}^q$BAB-process]\label{Exp-Sum} Assume $q\in \{1, 2\}$ if $d=3$ or $q\in \mathbb{N}$ if $d\geqslant 4$. Assume that $M>\max(1, \delta^{-1})$, dist(supp$(G)$, $\Omega^{c})\geqslant c'_1 \delta$ for some constant $c'_1$, and that for all $x\in\Omega$ and $\nu\in\mathbb{N}_{0}^{d}$ with $|\nu|\leqslant 3\lceil \frac{d}{2}\rceil+q+3$
\begin{equation}
(D^{\nu}G)(x)\lesssim \delta^{-|\nu|}\lesssim_{\delta} 1, \label{a}
\end{equation}
\begin{equation}
(D^{\nu}F)(x)\lesssim 1, \label{b}
\end{equation}

For $1\leqslant i, j\leqslant d$ denote
\begin{equation}
a_{i,j}^{(k)}(x)=\frac{\partial^{k+2}F}{\partial x_{1} \partial x_i \partial x_j \partial x_d^{k-1}
}(x).  \nonumber
\end{equation}

We further assume that for all $x\in\Omega$ and $k\in \{q, q+1\}$
\begin{equation}
|\textrm{det} (a_{i,j}^{(k)}(x))_{1\leqslant i,j\leqslant d}| \gtrsim 1 \label{c}
\end{equation}
and
\begin{equation}
\left\{
\begin{array}{ll}
|a_{i, i}^{(q)}(x)|\asymp 1   &\quad \textrm{for $1\leqslant i\leqslant d-1$} \\
|a_{i, j}^{(q)}(x)|\lesssim 1              &\quad \textrm{for $2\leqslant i\leqslant d-1$, $1\leqslant j\leqslant i-1$}\\
|a_{d, 1}^{(q)}(x)|\asymp 1                 &\quad\\
|a_{d, j}^{(q)}(x)|\lesssim \delta    &\quad \textrm{for $2\leqslant j\leqslant d$}
\end{array}.
\right. \label{entry-estimate}
\end{equation}

If $\delta$ is sufficiently small (only depending on $d$, $q$, and constants implied in \eqref{b}, \eqref{c}, and \eqref{entry-estimate}) and $T$ is restricted to
\begin{equation}
M^{q+\frac{2}{Q}-\frac{1}{d}-\frac{4}{d^2}} \leqslant T \leqslant
M^{q+\frac{2}{Q}+\frac{2}{d-2}} \quad (Q=2^q), \label{restriction}
\end{equation}
then
\begin{equation}
S(T, \delta M; G, F) \lesssim_{\delta} T^{\frac{d^2}{2(Q-1)d^2+2Qd+4Q}}
M^{d-\frac{(q+2)d^2+d}{2(Q-1)d^2+2Qd+4Q}}. \label{e}
\end{equation}
Besides $\delta$, the implicit constant in \eqref{e} depends on $d$, $q$, $c_0$, $c'_1$, and constants implied in \eqref{a}, \eqref{b}, \eqref{c}, and \eqref{entry-estimate}.
\end{theorem}

{\it Remarks:} 1) If $T\gtrsim M^{q+2+1/d}$, the trivial bound $S\lesssim M^d$ is better than above estimate.

2) In our later application of this theorem, we will let $q=1$ if $d\geqslant 4$ and $2$ if $d=3$; we will choose and fix a sufficiently small $\delta$ and we don't need it explicitly in the bound \eqref{e}.

\begin{proof}[Proof of Theorem \ref{Exp-Sum}]  Assume that $1<H\leqslant c_2\delta M$ with a small constant $c_2$ (to be determined later) and that $\delta M>c_2^{-1}$ (otherwise the trivial bound is better than \eqref{e}). Using Lemma \ref{weyl} with $r_1=e_1$ and $r_l=e_{d}$ ($2\leqslant l\leqslant q$), we get
\begin{equation}
\begin{split}
|S(T, \delta M; G, F)|^{Q}\lesssim &\frac{(\delta M)^{Qd}}{H} + \frac{(\delta M)^{(Q-1)d}}{H_1\cdots H_q} \cdot\\
&\sum_{\substack{1\leqslant h_i<H_i \\1 \leqslant i \leqslant q}} |S(\mathscr{H}T(\delta M)^{-q}, \delta M; G_q, F_q)|,
\end{split}\label{weyl-in-theorem}
\end{equation}
where $G_q$, $F_q$ are as defined in Lemma \ref{weyl}, so is the domain $\Omega_q$. Applying to the innermost sum the $d$-dimensional Poisson summation formula followed by a change of variables yields
\begin{align}
S(\mathscr{H}&T(\delta M)^{-q}, \delta M; G_q, F_q)\nonumber\\
&=\sum_{p\in \mathbb{Z}^d} (\delta M)^d \int
G_q(x) e\big(\mathscr{H}T(\delta M)^{-q}F_q(x)-\delta M x\cdot p\big) dx, \label{dddd}
\end{align}

Lemma \ref{weyl} and \ref{distance-boundary} imply
\begin{displaymath}
\textrm{supp}(G_q)\subset \Omega_q \subset c_0B(0, 1)
\end{displaymath}
and
\begin{equation}
\textrm{dist}(\textrm{supp}(G_q), \Omega_q^{c}) \geqslant c'_1 \delta. \nonumber
\end{equation}
By \eqref{b}, there exists a sufficiently large constant $A_0$ (independent of $\delta$) such that for any $x\in \Omega_q$
\begin{equation}
|DF_q(x)|\leqslant A_0/2.  \nonumber
\end{equation}
Define $\widetilde{M}=\mathscr{H} T (\delta M)^{-q-1}$. We split \eqref{dddd} into two parts
\begin{equation}
S(\mathscr{H}T(\delta M)^{-q}, \delta M; G_q, F_q)=\sum_{|p|\geqslant A_0 \widetilde{M}}+\sum_{|p|<A_0 \widetilde{M}}=:\textrm{I}+\textrm{II}. \nonumber
\end{equation}
and estimate them separately.

{\bf Case 1.} If $|p|\geqslant A_0 \widetilde{M}$. Like what we did in the proof of Proposition \ref{B-process}, it is easy to prove that sum $\textrm{I}\lesssim_{\delta} M^{-1}$.

{\bf Case 2.} If $|p|< A_0 \widetilde{M}$. Let $\widetilde{T}=\delta M \widetilde{M}$\footnote{The reason why we introduce new parameters $\widetilde{T}$ and $\widetilde{M}$ will be clear later.}.

{\bf Subcase 2.1} If $\widetilde{T}\geqslant 1$. Define
\begin{equation}
\Phi_q(x, z)=F_q(x)-x\cdot z \quad \textrm{$x\in \Omega_q$, $z\in \mathbb{R}^d$}, \nonumber
\end{equation}
then
\begin{equation}
\textrm{II}=(\delta M)^d \sum_{|p|<A_0 \widetilde{M}} \int G_q(x) e(\widetilde{T}
\Phi_q(x, \frac{p}{\widetilde{M}})) dx.  \nonumber
\end{equation}

For all $x\in \Omega_q$, $|z|<A_0$, and $|\nu|\leqslant 3\lceil \frac{d}{2}\rceil+3$, it is easy to see that
\begin{equation}
D_{x}^{\nu} G_q (x) \lesssim_{\delta} 1,  \label{resupbow}
\end{equation}
\begin{equation}
D_{x}^{\nu} \Phi_q(x, z) \lesssim 1,  \label{resupbo1}
\end{equation}
and if $c_2$ is sufficiently small (depending on constants implied in \eqref{b}, \eqref{c}) condition \eqref{c} with $k=q$ implies
\begin{equation}
\big|\textrm{det} (D^2_{xx}\Phi_q(x, z))\big|\gtrsim 1.  \label{i}
\end{equation}

Denote $f(x)=DF_q(x)$, then $
D_x \Phi_q(x, z)=f(x)-z$. If we can solve the following equation in $x$-variable for a particular $|z|<A_0$
\begin{equation}
f(x)=z \quad \textrm{for $x\in \Omega_q$}, \nonumber 
\end{equation}
we call the solution a critical point (with respect to $z$).

We know that $\textrm{supp}(G_q)$ is strictly smaller than $\Omega_q$ and the distance between their boundary is larger than $c'_1 \delta$. Let $r_0=c'_1 \delta/2$. By the Taylor's formula, there exists $r_*$ ($<r_0$) such that if $\widetilde{x}$ is a critical point in $(\textrm{supp}(G_q))_{(r_0)}$ with respect to $z$ then
\begin{equation}
|D_x\Phi_q(x, z)|\gtrsim |x-\widetilde{x}| \quad \textrm{for $x\in B(\widetilde{x}, r_*)$}. \label{claim2}
\end{equation}
The implicit constant depends on $d$, constants implied in \eqref{resupbo1}, \eqref{i}.

Applying Lemma \ref{app1} to $f$ with above $r_0$ yields two positive numbers $r_1$, $r_2$ (in particular both depending on $\delta$) such that $2r_1\leqslant r_*$ and for any $x\in (\textrm{supp}(G_q))_{(r_0)}$, $f$ is bijective from $B(x, 2r_1)$ to an open set containing $B(f(x), 2r_2)$. Note that $r_1<c'_1\delta/4$. If $x_1$, $x_2$ $\in (\textrm{supp}(G_q))_{(r_0)}$ are two different critical points with respect to $z$ (if exist), then $B(x_1, r_1)$ and $B(x_2, r_1)$ are disjoint and contained in $\Omega_q$.

\medskip

Next we will use two partitions of unity to restrict the domains for both $x$ and $z$ to small balls. We can choose finitely many balls $\{X_k\}_{k=1}^{K}$ and $\{Z_s\}_{s=1}^{S}$ (from families $\{B(x, r_1/3): x\in c_0B(0, 1)\}$ and $\{B(z, r_2/3): z\in (A_0/2) B(0, 1)\}$ respectively) and two families of functions $\{\phi_k\}_{k=1}^{K}$ and $\{\eta_s\}_{s=1}^{S}$ such that 
\begin{enumerate}
\item  $c_0B(0, 1)\subset \cup_{k=1}^{K} X_k$ and $(A_0/2) B(0, 1)\subset \cup_{s=1}^{S} Z_s$;
\item  $K$ and $S$ are both bounded above by some constants;
\item  $\sum_{k=1}^{K}\phi_k(x)\equiv 1$ if $x \in c_0B(0, 1)$ and $\phi_k \in C_0^{\infty}(X_k)$;
\item  $\sum_{s=1}^{S}\eta_s(z)\equiv 1$ if $z \in (A_0/2) B(0, 1)$ and $\eta_s \in C_0^{\infty}(Z_s)$.
\end{enumerate}
Denote $\eta_0=1-\sum_{s=1}^{S}\eta_s$. Adding these cut-off functions, we get
\begin{equation}
\textrm{II}=(\delta M)^d \sum_{k=1}^{K}\sum_{s=0}^{S}\textrm{III}(k, s),  \nonumber
\end{equation}
where
\begin{equation}
\textrm{III}(k, s)=\sum_{|p|<A_0 \widetilde{M}} \eta_s\big(\frac{p}{\widetilde{M}}\big)\int \phi_k(x) G_q(x) e(\widetilde{T}
\Phi_q(x, \frac{p}{\widetilde{M}})) dx.  \label{bb}
\end{equation}

Let's fix arbitrarily $0\leqslant s\leqslant S$, $1\leqslant k\leqslant K$ and estimate this sum III. Denote $E_k=\textrm{supp}(\phi_k) \cap \textrm{supp}(G_q)$. We will only consider those $k$'s such that $E_k\neq \emptyset$, otherwise the integrals above vanish.

For $|z|<A_0$ define
\begin{equation}
S_z=\{x\in E_k: |D_x\Phi_q(x, z)|<r_2/3\}. \nonumber
\end{equation}

If $S_z$ is empty for a $z$, $|D_x\Phi_q(x, z)|$ has a lower bound $r_2/3$ on $E_k$. As a consequence, for some $p$ with nonempty $S_{p/\widetilde{M}}$ the integral in sum III$(k, s)$ is of order $O_{\delta}(\widetilde{T}^{-\lceil \frac{d}{2}\rceil-1})$ by integration by parts.

If $S_z$ is not empty for a $z$, Lemma \ref{app1} ensures that there exists a unique critical point $x(z)$ in $(E_k)_{(r_1/3)}\subset \Omega_q$. 

If $s=0$, we actually sum over all integral $p$'s such that $A_0\widetilde{M}/2<|p|< A_0\widetilde{M}$. For those $p$'s, $D_x\Phi_q(x, p/\widetilde{M})\neq 0$ for $x\in\Omega_q$. It follows that $S_{p/\widetilde{M}}$ is empty, hence each integral in \eqref{bb} is of order $O_{\delta}(\widetilde{T}^{-\lceil \frac{d}{2}\rceil-1})$. Thus III$(k, 0)$ $\lesssim_{\delta}(1+\widetilde{M}^d) \widetilde{T}^{-\lceil \frac{d}{2}\rceil-1}$.

Now let's assume $s\geqslant 1$. Since $\eta_s$ is compactly supported we can replace the summation domain in sum III by $\{p\in \mathbb{Z}^d\}$. Assume there exists a $p_1\in \mathbb{Z}^d$ such that $\eta_s(p_1/\widetilde{M})\neq 0$ and $S_{p_1/\widetilde{M}}$ is not empty. Hence the critical point $x(p_1/\widetilde{M})$ exists in $(E_k)_{(r_1/3)}$. It follows that for every $z\in B(p_1/\widetilde{M}, r_2)$, critical point $x(z)$ exists in $B(x(p_1/\widetilde{M}), r_1)$ and is smooth on $B(p_1/\widetilde{M}, r_2)$. Since $\textrm{supp} (\eta_s)$ $\subset Z_s \subset B(p_1/\widetilde{M}, 2r_2/3)$, $\textrm{dist}$$\{\textrm{supp}(\eta_s)$, we have $B(p_1/\widetilde{M}, r_2)^{c} \}$$\geqslant r_2/3$.

We also have $E_k\subset B(x(z), 2r_1)\subset \Omega_q$ for any $z\in B(p_1/\widetilde{M}, r_2)$. Recalling the \eqref{claim2} and applying Lemma \ref{SP}\footnote{The $K$, $X$ in that lemma can be chosen to be $E_k$, $B(x(p/\widetilde{M}), 2r_1)$ respectively.} yield
\begin{equation}
\textrm{III}(k, s)=\widetilde{T}^{-\frac{d}{2}} S(\widetilde{T}, \widetilde{M}; \widetilde{G}, \widetilde{F})+O_{\delta}(\sum_{p\in \mathbb{Z}^d} \eta_s(\frac{p}{\widetilde{M}})\widetilde{T}^{-\frac{d}{2}-1}), \nonumber
\end{equation}
where
\begin{equation}
\widetilde{G}(z)=\eta_s(z) \phi_k(x(z)) G_q(x(z)) |\textrm{det} (Q(z))|^{-\frac{1}{2}} ,\nonumber
\end{equation}
\begin{equation}
\widetilde{F}(z)=\Phi_q(x(z), z)+\textrm{sgn} (Q(z))/(8\widetilde{T}), \nonumber
\end{equation}
and $Q(z)=D_{xx}^{2}\Phi_q(x(z),z)$. Denote the domain of $\widetilde{G}$ by $\mathscr{D}$, whose possible choice is $B(p_1/\widetilde{M}, r_2)$. It satisfies $\textrm{supp}(\widetilde{G}) \subset \mathscr{D} \subset A_0B(0, 1)$ and $\textrm{dist}\{\textrm{supp}(\widetilde{G}), \mathscr{D}^{c} \}\geqslant r_2/3$.

Now we need to estimate the new exponential sum $S(\widetilde{T}, \widetilde{M}; \widetilde{G}, \widetilde{F})$. We first make the following claim:
\begin{claim}\label{verify-cond-thm}For all $z\in\mathscr{D}$ and $|\nu|\leqslant 3\lceil \frac{d}{2}\rceil+2$, the following bounds
\begin{equation}
(D^{\upsilon} \widetilde{G})(z) \lesssim_{\delta} 1, \nonumber
\end{equation}
\begin{equation}
(D^{\upsilon} \widetilde{F})(z) \lesssim 1 \nonumber
\end{equation}
hold. Furthermore, if $\delta$ and $c_2$ are sufficiently small (both depending on $d$ and constants implied in \eqref{b}, \eqref{c}, and \eqref{entry-estimate}), then
\begin{equation}
|\textrm{det} \big(D^3_{1,i,j}\widetilde{F}(z)\big)_{1\leqslant i,j\leqslant d}|\gtrsim 1. \nonumber
\end{equation}
In particular, all three constants implied in these bounds are independent of the choice of domain $\mathscr{D}$.
\end{claim}
We defer the proof of this claim until later.

\medskip

If $\widetilde{M}\geqslant 1$ assume $c_2$ is sufficiently small (depending on $d$, $q$, and $\delta$), then the assumption $T \leqslant M^{q+2/Q+2/(d-2)}$ implies
\begin{displaymath}
\widetilde{T}\geqslant \widetilde{M}^{2-\frac{2}{d}}.
\end{displaymath}
Hence we are allowed to apply Proposition \ref{AB-process} (The $q$, $\mu$ there can be taken to be $1$, $e_1$ respectively.) and get
\begin{equation}
S(\widetilde{T}, \widetilde{M}; \widetilde{G}, \widetilde{F}) \lesssim_{\delta} \widetilde{T}^{\frac{d}{2(d+2)}}\widetilde{M}^{d-\frac{3d}{2(d+2)}}.
\nonumber
\end{equation}

If $\widetilde{M}\leqslant 1$, the trivial estimate gives
\begin{equation}
S(\widetilde{T}, \widetilde{M}; \widetilde{G}, \widetilde{F})\lesssim 1.\nonumber
\end{equation}

Combining these two bounds, we get
\begin{equation}
S(\widetilde{T}, \widetilde{M}; \widetilde{G}, \widetilde{F}) \lesssim_{\delta} 1+\widetilde{T}^{\frac{d}{2(d+2)}}\widetilde{M}^{d-\frac{3d}{2(d+2)}}.
\nonumber
\end{equation}
Finally, we get the bound for sum II in this subcase:
\begin{equation}
\textrm{II}  \lesssim_{\delta} (\delta M)^d [\widetilde{T}^{-\frac{d}{2}}(1+\widetilde{T}^{\frac{d}{2(d+2)}}\widetilde{M}^{d-\frac{3d}{2(d+2)}})+(1+\widetilde{M}^d)(\widetilde{T}^{-\frac{d}{2}-1}+\widetilde{T}^{-\lceil \frac{d}{2}\rceil-1})] \nonumber
\end{equation}
\begin{equation}
\lesssim_{\delta} M^{\frac{d^2+3d}{2(d+2)}}\widetilde{M}^{\frac{d^2}{2(d+2)}}+M^{\frac{d}{2}-1}\widetilde{M}^{\frac{d}{2}-1}
+M^{\frac{d}{2}}\widetilde{M}^{-\frac{d}{2}},\nonumber
\end{equation}
\begin{equation}
\lesssim_{\delta} M^{\frac{d^2+3d}{2(d+2)}}\widetilde{M}^{\frac{d^2}{2(d+2)}}+M^{\frac{d}{2}}\widetilde{M}^{-\frac{d}{2}}.\label{esti-II}
\end{equation}

In the second inequality, we use $\widetilde{T}=\delta M \widetilde{M}\geqslant 1$. In the last inequality, we omit the second term since
\begin{displaymath}
M^{\frac{d^2+3d}{2(d+2)}}\widetilde{M}^{\frac{d^2}{2(d+2)}}=M^{\frac{3d}{2(d+2)}}(M\widetilde{M})^{\frac{d^2}{2(d+2)}}
\gtrsim_{\delta}(M\widetilde{M})^{\frac{d}{2}-1}.
\end{displaymath}

{\bf Subcase 2.2} If $\widetilde{T}<1$, which implies $\delta M<\widetilde{M}^{-1}$ and $\widetilde{M}<1$. Hence
\begin{equation}
\textrm{II}\lesssim_{\delta} (\delta M)^d\lesssim_{\delta} M^{\frac{d}{2}}\widetilde{M}^{-\frac{d}{2}}.\nonumber
\end{equation}

Comparing this bound with the bound \eqref{esti-II}, we conclude that \eqref{esti-II} always holds for sum II.

Using the bounds for sum I and II, we get
\begin{equation}
S(\mathscr{H}T(\delta M)^{-q}, \delta M; G_q, F_q)  \lesssim_{\delta} M^{-1}+M^{\frac{d^2+3d}{2(d+2)}}\widetilde{M}^{\frac{d^2}{2(d+2)}}+M^{\frac{d}{2}}\widetilde{M}^{-\frac{d}{2}}\nonumber
\end{equation}
\begin{equation}
\lesssim_{\delta} \mathscr{H}^{\frac{d^2}{2(d+2)}}T^{\frac{d^2}{2(d+2)}}M^{\frac{-qd^2+3d}{2(d+2)}} +\mathscr{H}^{-\frac{d}{2}}T^{-\frac{d}{2}}M^{\frac{q+2}{2}d}. \nonumber
\end{equation}
In the last step, we use definition of $\widetilde{M}$ and omit $M^{-1}$ since it is smaller than the sum of the other two no matter whether $\widetilde{M}\geqslant 1$ or $<1$.

Plugging this bound into \eqref{weyl-in-theorem} yields
\begin{equation}
\begin{split}
|S(T, \delta M; G, &F)|^{Q} \lesssim_{\delta} M^{Qd}\big(H^{-1}+\\
&M^{-\frac{(q+2)d^2+d}{2(d+2)}}
T^{\frac{d^2}{2(d+2)}}
H^{\frac{(1-1/Q)d^2}{d+2}}+T^{-\frac{d}{2}}M^{\frac{q}{2}d}H^{-2+2/Q}\big).\nonumber
\end{split}
\end{equation}
Balancing the first two terms yields the optimal choice
\begin{equation}
H=B_3 T^{-\frac{d^2}{(2-2/Q)d^2+2d+4}} M^{\frac{(q+2)d^2+d}{(2-2/Q)d^2+2d+4}}, \nonumber
\end{equation}
where $B_3$ can be chosen so small that $T\geqslant M^{q+2/Q-1/d-4/d^2}$ implies $H\leqslant c_2 \delta M$. Due to Remark 1, we can assume $T\leqslant B_4 M^{q+2+1/d}$ with a sufficiently small $B_4$, which implies
$1< H$.

For $q\in \{1, 2\}$ if $d=3$ or $q\in \mathbb{N}$ if $d\geqslant 4$, we have
\begin{equation}
T^{-\frac{d}{2}}M^{\frac{q}{2}d}H^{-2+2/Q}\lesssim_{\delta} M^{-d-\frac{1}{2}}H^{d-\frac{1}{2}} \lesssim_{\delta} H^{-1}. \nonumber
\end{equation}
Hence
\begin{equation}
S(T, \delta M; G, F) \lesssim_{\delta} T^{\frac{d^2}{2(Q-1)d^2+2Qd+4Q}}
M^{d-\frac{(q+2)d^2+d}{2(Q-1)d^2+2Qd+4Q}}. \nonumber
\end{equation}
This finishes the proof of this theorem.
\end{proof}

\begin{proof} [Proof of Claim \ref{verify-cond-thm}] Let's consider $z\in \mathscr{D}$. The critical point function $x(z)=(x_1(z), \ldots, x_d(z))$ satisfies the equation $D_x \Phi_q(x(z), z)=0$, namely
\begin{equation}
D_x F_q(x(z))-z=0. \nonumber
\end{equation}

Differentiating this equation gives
\begin{equation}
D_{xx}^2F_q(x(z)) D_z x(z)-I_d=0  \nonumber
\end{equation}
where $I_d$ is the unit matrix of size $d$, hence
\begin{equation}
D_z x(z)=(D_{xx}^2F_q(x(z)))^{-1}.
\label{l}
\end{equation}

By differentiating this formula inductively and using bounds \eqref{b}, \eqref{i} for $F_q$, we get
\begin{equation}
D_z^{\upsilon} x_{i}(z) \lesssim 1 \quad
\textrm{for $1\leqslant i \leqslant d$, $z\in\mathscr{D}$, and $|\nu|\leqslant 3\lceil \frac{d}{2}\rceil+2$}. \nonumber
\end{equation}

This bound together with the chain rule and product rule gives us the two upper bounds in the claim. To prove the lower bound of det$(D^3_{1,i,j}\widetilde{F})$, we first have
\begin{align}
D_z \widetilde{F}(z) &=D_z\big(F_q(x(z))-x(z)\cdot z+\textrm{sgn} (Q(z))/(8\widetilde{T})\big)  \nonumber\\
      &=-x(z)+D_z x(z)[D_x F_q(x(z))-z]=-x(z). \nonumber
\end{align}
Derivative of $\textrm{sgn} (Q(z))$ vanishes since it is a constant function and the last equality follows from the defining equation of critical points. Thus
\begin{align}
\big(D^3_{1,i,j}\widetilde{F}(z)\big)_{1\leqslant i,j\leqslant d} &=
-\frac{\partial}{\partial z_{1}}\big( D_z x(z)\big) \nonumber \\
&=(D_{xx}^2F_q)^{-1}\frac{\partial}{\partial
z_{1}}(D_{xx}^2F_q(x(z)))
(D_{xx}^2F_q)^{-1}. \nonumber
\end{align}
In last step we use \eqref{l}. Hence we get the desired lower bound if we can prove
\begin{equation}
|\textrm{det} \big(\frac{\partial}{\partial
z_{1}}(D_{xx}^2F_q(x(z)))\big)|\gtrsim 1. \label{anothermatrix}
\end{equation}

If $\delta$ is sufficiently small and $H\leqslant \delta^2 M$, condition \eqref{entry-estimate} ensures the following bounds for entries of the symmetric matrix $D_{xx}^2F_q$: $|D^{2}_{i, i}F_q|\asymp 1$ for $1\leqslant i\leqslant d-1$; $D^{2}_{i, j}F_q\lesssim 1$ for $2\leqslant i\leqslant d-1$, $1\leqslant j\leqslant i-1$; $|D^{2}_{d, 1}F_q|\asymp 1$; $D^{2}_{d, j}F_q\lesssim \delta$ for $2\leqslant j\leqslant d$.

We can then estimate the entries of $D_z x(z)=(D_{xx}^2F_q(x(z)))^{-1}$. Actually, we only need to consider the first column of $D_z x(z)$. We have $\frac{\partial x_i}{\partial z_{1}} \lesssim \delta$ for $1\leqslant i\leqslant d-1$. If $\delta$ is sufficiently small, we have $|\frac{\partial x_d}{\partial z_{1}}| \asymp 1$.

If $H\leqslant c_3 \delta M$ with a sufficiently small $c_3$ (depending on $d$ and constants implied in \eqref{b}, \eqref{c}), condition \eqref{c} with $k=q+1$ implies
\begin{equation}
\big|\textrm{det}\big(D^3_{i, j, d} F_q \big)_{1\leqslant i,j\leqslant d}\big| \gtrsim 1. \nonumber
\end{equation}

Note that
\begin{equation}
\frac{\partial}{\partial
z_{1}}(D_{xx}^2F_q(x(z)))
=\sum_{l=1}^d \big(D^3_{i, j, l} F_q(x(z)) \big)_{1\leqslant i,j\leqslant d}
\frac{\partial x_l(z)}{\partial z_{1}}. \nonumber
\end{equation}

If $\delta$ is sufficiently small, then $\frac{\partial x_l}{\partial z_{1}}$'s ($1\leqslant l\leqslant d-1$)
are relatively smaller than $\frac{\partial x_d}{\partial z_{1}}$. The terms with $l=d$ in above summation overweigh the others and this leads to \eqref{anothermatrix}.

$\delta$ only depends on $d$ and constants implied in \eqref{b}, \eqref{c}, and \eqref{entry-estimate}. We require $c_2$ to be smaller than $\delta$ and $c_3$, and it depends on the same quantities as $\delta$ does. From the argument we can see that all bounds are independent of the choice of $\mathscr{D}$.
\end{proof}

\section{Nonvanishing of $d\times d$ determinants}\label{nonvanishing}
In this section, we will give lower bounds of determinants of certain $d\times d$ matrices and
description of sizes of their entries. These results are obtained based on M\"{u}ller \cite{mullerII} Lemma 3 and its proof.

For $\xi\neq 0$, let $H(\xi)=\sup_{x\in \mathcal{B}}\langle \xi, x\rangle$. It is a real-valued function positively homogeneous of degree $1$, i.e. $H(k \xi)=kH(\xi)$ if $k>0$. Due to the curvature condition imposed on $\partial \mathcal{B}$, $H$ is smooth and the eigenvalues of Hessian matrix of $H$ at $\xi\neq 0$ are $0$ and $(d-1)$ real numbers comparable to $1/|\xi|$. This simple fact is not hard to prove and the reader can also check \cite{Bonnesen}.

Given any $d$ vectors $v_1, \ldots, v_d$, by writing $V=(v_1, \ldots, v_d)$ we mean $V$ is the matrix with column vectors $v_1, \ldots, v_d$. If $y\neq 0$ we define $F(u_1, \ldots, u_d)=H(y+\sum_{l=1}^{d}u_l v_l)$, $u_l\in \mathbb{R}$ $(1\leqslant l\leqslant d)$. For $1\leqslant i, j \leqslant d$ and $k\in \mathbb{N}$, define
\begin{equation}
g_{i, j}^{(k)}(y, v_1, \ldots, v_d)=\frac{\partial^{k+2}F}{\partial
u_1 \partial u_i \partial u_j \partial u_d^{k-1}}(0), \nonumber
\end{equation}
which form a symmetric matrix
\begin{equation}
G_{k}(y, v_1, \ldots, v_d)=\big(g_{i, j}^{(k)}(y, v_1, \ldots, v_d)\big)_{1\leqslant i, j\leqslant d}\nonumber
\end{equation}
with determinant
\begin{equation}
h_k(y, v_1, \ldots, v_d)=\textrm{det}(G_{k}(y, v_1, \ldots, v_d)). \nonumber
\end{equation}

Denote
\begin{equation}
\mathscr{C}_1=\{x\in \mathbb{R}^d: 1/2 \leqslant |x| \leqslant 2 \}, \nonumber
\end{equation}
\begin{equation}
\mathscr{C}_1^{+}=\{x\in \mathbb{R}^d: 1/4 \leqslant |x| \leqslant 4 \}. \nonumber
\end{equation}
Since $H$ is smooth we can assume its derivatives up to order $q+3$ on $\mathscr{C}_1^{+}$ are bounded by a constant (only depending on $q$ and $\mathcal{B}$):
\begin{equation}
D^{\nu}H(\xi)\lesssim 1 \quad \textrm{for all $\xi\in \mathscr{C}_1^{+}$ and $|\nu|\leqslant q+3$}.\label{final-depend}
\end{equation}
We will only consider points in $\mathscr{C}_1$ in the following lemma.

\begin{lemma}\label{keylemma} Assume $q$ and $N$ are both positive integers. There exists $A_3>0$ (depending on $q$ and $\mathcal{B}$) such that if $N\geqslant A_3$ then for every $\xi\in \mathscr{C}_1$ there exist $d$
linearly independent vectors $v_1, \ldots, v_d\in \mathbb{Z}^d$
$($depending on $\xi)$ such that $|v_l|\asymp N$ $(1\leqslant l\leqslant d)$, $|\textrm{det}(V)|\asymp N^d$, and for every $1\leqslant k\leqslant q$ and $y\in B(\xi, 1/N)$
\begin{equation}
|h_k(y, v_1, \ldots, v_d)|\gtrsim N^{(k+2)d}   \nonumber
\end{equation}
and
\begin{equation}
\left\{
\begin{array}{ll}
|g_{i, i}^{(k)}(y, v_1, \ldots, v_d)|\asymp N^{k+2}   & \textrm{for $1\leqslant i\leqslant d-1$} \\
|g_{i, j}^{(k)}(y, v_1, \ldots, v_d)|\lesssim N^{k+2}              & \textrm{for $2\leqslant i\leqslant d-1$, $1\leqslant j\leqslant i-1$}\\
|g_{d, 1}^{(k)}(y, v_1, \ldots, v_d)|\asymp N^{k+2}                 \\
|g_{d, j}^{(k)}(y, v_1, \ldots, v_d)|\lesssim N^{k+1}     & \textrm{for $2\leqslant j\leqslant d$}
\end{array}. \right. \nonumber
\end{equation}

All implicit constants may depend on $q$ and $\mathcal{B}$.
\end{lemma}

\begin{proof}[Proof of Lemma \ref{keylemma}] We will essentially follow the proof of M\"{u}ller \cite{mullerII} Lemma 3 (with some minor modification) and establish these results through three steps for an arbitrarily fixed $\xi\in \mathscr{C}_1$.

{\bf Step 1.} We first choose $d$ vectors $P_l\in \mathbb{R}^d$ ($1\leqslant l \leqslant d$), in particular $P_1=\xi$, such that $|P_l|=|\xi|$ and $P_l/|\xi|$'s form an orthogonal matrix. Let $P=(P_1, \ldots, P_d)$ and $\widetilde{H}(y)=H(Py)$. $\widetilde{H}$ is positively homogeneous of degree $1$ and the eigenvalues of Hessian matrix of $\widetilde{H}$ at $e_1$ are $0$ and ($d-1$) real numbers comparable to $1$ since $D^2 \widetilde{H}(e_1)$ is similar to $D^2 H(\xi)$ up to a number $|\xi|^2$ and $|\xi|\asymp 1$.

Set $A=(\widetilde{H}_{ij}(e_1))$. $A$ is a symmetric matrix of rank $d-1$ with vanishing first row and column (due to homogeneity, cf. proof of Lemma 3 in M\"{u}ller \cite{mullerII}).
Choose a system of orthonormal eigenvectors $w_1', \ldots, w_{d-1}'$ of $A$, whose first components all vanish. Denote
the eigenvalue of $w_i'$ by $\lambda_i$ (comparable to $1$). Note that for every $\alpha>1$ the vector $w_1=w_1'+\alpha e_1$ is orthogonal to $w_j'$ for $2\leqslant j\leqslant d-1$ and satisfies $A w_1=\lambda_1 w_1'$. Denote
\begin{displaymath}
w_i=\left\{
\begin{array}{ll}
w_1'+\alpha e_1   &\quad \textrm{if $i=1$} \\
w_i'              &\quad \textrm{if $2\leqslant i\leqslant d-1$}\\
e_1              &\quad \textrm{if $i=d$}
\end{array}, \right.
\end{displaymath}
then $|w_1|\asymp \alpha$, $|w_l|=1$ ($2\leqslant l\leqslant d$), and $\textrm{det}(W)=1$ where $W=(w_1, \ldots, w_d)$. Denote $w_i=(w_{i,1}, \ldots, w_{i,d})^{t}$, $F(u_1, \ldots, u_d)=\widetilde{H}(e_1+\sum_{l=1}^{d}u_l w_l)$, and
\begin{displaymath}
b_{i,j}^{(k)}(\alpha)=\frac{\partial^{k+2}F}{\partial
u_1 \partial u_i \partial u_j \partial u_d^{k-1}}(0).
\end{displaymath}

Define $v_l^*=Pw_l$. Then $|v_1^*|\asymp \alpha$, $|v_l^*|\asymp 1$ ($2\leqslant l\leqslant d$), and $|\textrm{det}(V^*)|\asymp 1$ where $V^*=(v_1^*, \ldots, v_d^*)$. Note that $F(u_1, \ldots, u_d)=H(\xi+\sum_{l=1}^{d}u_l v_l^*)$ and $b_{i,j}^{(k)}(\alpha)=g_{i, j}^{(k)}(\xi, v_1^{*}, \ldots, v_d^{*})$.

If $1\leqslant i, j\leqslant d-1$,
\begin{displaymath}
b_{i,j}^{(k)}(0)=\sum_{m, n, s=1}^{d}\frac{\partial^{k+2}\widetilde{H}}{\partial
y_1^{k-1} \partial y_m \partial y_n \partial y_s}(e_1)w'_{1,m}w'_{i,n}w'_{j,s}\lesssim 1.
\end{displaymath}
The last inequality is due to assumption \eqref{final-depend}.

If $i=1$, $1\leqslant j\leqslant d-1$, then
\begin{displaymath}
b_{1,j}^{(k)}(\alpha)=b_{1,j}^{(k)}(0)+3\alpha (-1)^k k!\lambda_1 \delta_{1j},
\end{displaymath}
where $\delta_{ij}$ is the Kronecker notation.

If $2\leqslant i, j\leqslant d-1$, then
\begin{displaymath}
b_{i,j}^{(k)}(\alpha)=b_{i,j}^{(k)}(0)+\alpha (-1)^k k!\lambda_j \delta_{ij}.
\end{displaymath}

If $1\leqslant i\leqslant d$, $j=d$, then
\begin{displaymath}
b_{i,d}^{(k)}(\alpha)=(-1)^k k!\lambda_1 \delta_{1i}.
\end{displaymath}

Using these formulas, we get
\begin{displaymath}
\textrm{det}(b_{i,j}^{(k)}(\alpha))_{1\leqslant i, j\leqslant d}=-(k!\lambda_1)^2 \textrm{det}(b_{i,j}^{(k)}(\alpha))_{2\leqslant i, j\leqslant d-1}
\end{displaymath}
and
\begin{displaymath}
\textrm{det}(b_{i,j}^{(k)}(\alpha))_{2\leqslant i, j\leqslant d-1}=\textrm{det}(b_{i,j}^{(k)}(0)+\alpha (-1)^k k!\lambda_j \delta_{ij})_{2\leqslant i, j\leqslant d-1}.
\end{displaymath}

The last determinant is a polynomial in $\alpha$ of degree $d-2$ with leading coefficient comparable to $1$. If we fix $\alpha$ to be a sufficiently large constant (only depending on $q$ and $\mathcal{B}$), then
\begin{displaymath}
|h_k(\xi, v_1^*, \ldots, v_d^*)|=|\textrm{det}(b_{i,j}^{(k)}(\alpha))_{1\leqslant i, j\leqslant d}|\gtrsim 1 \quad \textrm{for $1\leqslant k\leqslant q$}
\end{displaymath}
and
\begin{displaymath}
\left\{
\begin{array}{ll}
|g_{i, i}^{(k)}(\xi, v_1^{*}, \ldots, v_d^{*})|\asymp 1  &\quad \textrm{for $1\leqslant i\leqslant d-1$} \\
|g_{i, j}^{(k)}(\xi, v_1^{*}, \ldots, v_d^{*})|\lesssim 1              &\quad \textrm{for $2\leqslant i\leqslant d-1$, $1\leqslant j\leqslant i-1$}\\
|g_{d, 1}^{(k)}(\xi, v_1^{*}, \ldots, v_d^{*})|\asymp 1                 \\
|g_{d, j}^{(k)}(\xi, v_1^{*}, \ldots, v_d^{*})|=0     &\quad \textrm{for $2\leqslant j\leqslant d$}
\end{array}, \right.
\end{displaymath}
where the implicit constants only depend on $q$ and $\mathcal{B}$.

{\bf Step 2.} There exist vectors $v_l^{**}\in \mathbb{Q}^d$ ($1\leqslant l\leqslant d$), whose components are all ratios of two integers with denominator $N$, such that $|v_l^{**}-v_l^ {*}|\leqslant \sqrt{d}/N$. There exists a large number $A_1$ (only depending on $q$ and $\mathcal{B}$) such that if $N\geqslant A_1$ then $|v_l^{**}|\asymp 1$ ($1\leqslant l\leqslant d$) and $|\textrm{det}(V^{**})|\asymp 1$ where $V^{**}=(v_1^{**}, \ldots, v_d^{**})$. Since
\begin{equation}
|g_{i, j}^{(k)}(\xi, v_1^{**}, \ldots, v_d^{**})-g_{i, j}^{(k)}(\xi, v_1^{*}, \ldots, v_d^{*})|\lesssim 1/N, \nonumber
\end{equation}
there exists a large number $A_2\geqslant A_1$ (only depending on $q$ and $\mathcal{B}$) such that if $N\geqslant A_2$ then
\begin{equation}
|h_k(\xi, v_1^{**}, \ldots, v_d^{**})|\gtrsim 1 \quad \textrm{for $1\leqslant k\leqslant q$}\nonumber
\end{equation}
and
\begin{displaymath}
\left\{
\begin{array}{ll}
|g_{i, i}^{(k)}(\xi, v_1^{**}, \ldots, v_d^{**})|\asymp 1   &\quad \textrm{for $1\leqslant i\leqslant d-1$} \\
|g_{i, j}^{(k)}(\xi, v_1^{**}, \ldots, v_d^{**})|\lesssim 1              &\quad \textrm{for $2\leqslant i\leqslant d-1$, $1\leqslant j\leqslant i-1$}\\
|g_{d, 1}^{(k)}(\xi, v_1^{**}, \ldots, v_d^{**})|\asymp 1                 \\
|g_{d, j}^{(k)}(\xi, v_1^{**}, \ldots, v_d^{**})|\lesssim 1/N     &\quad \textrm{for $2\leqslant j\leqslant d$}
\end{array}, \right.
\end{displaymath}
where the implicit constants only depend on $q$ and $\mathcal{B}$.

{\bf Step 3.} Let $v_l=Nv_l^{**}$. Then $v_l\in \mathbb{Z}^d\setminus \{0\}$, $|v_l|\asymp N$ ($1\leqslant l\leqslant d$), and $|\textrm{det}(V)| \asymp N^d$. Note that
\begin{displaymath}
g_{i, j}^{(k)}(\xi, v_1, \ldots, v_d)=N^{k+2} g_{i, j}^{(k)}(\xi, v_1^{**}, \ldots, v_d^{**})
\end{displaymath}
Applying the mean value theorem, we have for $y\in \mathscr{C}_1^{+}$
\begin{equation}
|g_{i, j}^{(k)}(y, v_1^{**}, \ldots, v_d^{**})-g_{i, j}^{(k)}(\xi, v_1^{**}, \ldots, v_d^{**})|\lesssim |y-\xi|. \nonumber
\end{equation}
Thus there exists a large number $A_3\geqslant A_2$ (only depending on $q$ and $\mathcal{B}$) such that if $N\geqslant A_3$ and $y\in B(\xi, 1/N)$ then the bounds for determinants and entries in the lemma are both true and the implicit constants only depend on $q$ and $\mathcal{B}$. This finishes the proof.
\end{proof}

\section{Proof of Theorem \ref{lattice}}\label{mainproof}

By a standard procedure, we can change the combinatorial problem of counting lattice points in a blown-up domain to an analytical problem. The essential issue will be reduced to the estimation of an exponential sum. In order to apply Theorem \ref{Exp-Sum}, we need to introduce a dyadic decomposition and a partition of unity.

\begin{proof}[Proof of Theorem \ref{lattice}] Assume $\rho$ is a smooth function on $\mathbb{R}^d$ with compact support, which satisfies $\int_{\mathbb{R}^d}\rho(y)dy=1$. Let $\varepsilon$ be a small positive number, $\rho_{\varepsilon}(y)=
\varepsilon^{-d}\rho(\varepsilon^{-1}y)$, and
\begin{equation}
N_{\varepsilon}(t)=\sum_{k\in \mathbb{Z}^d} \chi_{t\mathcal{B}}*\rho_{\varepsilon}(k), \nonumber
\end{equation}
where $\chi_{t\mathcal{B}}$ denotes the characteristic function of domain $t\mathcal{B}$. By the Poisson summation formula
\begin{equation}
N_{\varepsilon}(t)=t^d\sum_{k\in \mathbb{Z}^d} \hat{\chi}_{\mathcal{B}}(tk)\hat{\rho}(\varepsilon k)=\textrm{vol}(\mathcal{B}) t^d+R_{\varepsilon}(t),  \nonumber
\end{equation}
where
\begin{equation}
R_{\varepsilon}(t)=t^d\sum_{k\in \mathbb{Z}_{*}^{d}} \hat{\chi}_{\mathcal{B}}(tk)\hat{\rho}(\varepsilon k). \nonumber
\end{equation}

M\"uller proved in \cite{mullerI} that there exists a constant $C_1$ such that
\begin{equation}
N_{\varepsilon}(t-C_1\varepsilon)\leqslant \#(t\mathcal{B}\cap\mathbb{Z}^d)=\sum_{k\in \mathbb{Z}^d} \chi_{t\mathcal{B}}(k) \leqslant N_{\varepsilon}(t+C_1\varepsilon), \nonumber
\end{equation}
which implies
\begin{equation}
P_{\mathcal{B}}(t)\lesssim |R_{\varepsilon}(t+C_1\varepsilon)|+|R_{\varepsilon}(t-C_1\varepsilon)|+t^{d-1}\varepsilon. \label{basic}
\end{equation}

It suffices to estimate $R_{\varepsilon}(t)$ for any large $t$. By H\"ormander \cite{hormander} Corollary 7.7.15, we have the asymptotic
expansion
\begin{equation}
\hat{\chi}_{\mathcal{B}}(\xi)=[C K_{\xi}^{-\frac{1}{2}}e^{-2\pi iH(\xi)}+C' K_{-\xi}^{-\frac{1}{2}}e^{2\pi iH(-\xi)}]|\xi |^{-\frac{d+1}{2}}+O(|\xi|^{-\frac{d+3}{2}}),  \nonumber
\end{equation}
where $C$, $C'$ are two constants, $H(\xi)=\sup_{x\in \mathcal{B}}\langle \xi, x\rangle$, and $K_{\xi}$ is the curvature at the boundary point where the exterior normal is $\xi$. $K_{\xi}$ is smooth on $\mathbb{R}^d\setminus \{0\}$ and positively homogeneous of degree $0$. Applying this formula gives
\begin{equation}
R_{\varepsilon}(t)=CS_1+C'\widetilde{S}_1+\textrm{Error},\nonumber
\end{equation}
where
\begin{equation}
S_1=t^{\frac{d-1}{2}} \sum_{k\in \mathbb{Z}_{*}^{d}}|k|^{-\frac{d+1}{2}}K_{k}^{-\frac{1}{2}}\hat{\rho}(\varepsilon k)e(tH(k)),\nonumber
\end{equation}
\begin{equation}
\widetilde{S}_1=t^{\frac{d-1}{2}} \sum_{k\in \mathbb{Z}_{*}^{d}}|k|^{-\frac{d+1}{2}}K_{-k}^{-\frac{1}{2}}\hat{\rho}(\varepsilon k)e(-tH(-k)),\nonumber
\end{equation}
and
\begin{equation}
\textrm{Error}\lesssim t^{\frac{d-3}{2}}\sum_{k\in \mathbb{Z}_{*}^{d}} |k|^{-\frac{d+3}{2}}\hat{\rho}(\varepsilon k)\lesssim t^{\frac{d-3}{2}}\varepsilon^{-\frac{d-3}{2}}.\label{err}
\end{equation}

Since the first two sums are
similar, it suffices to estimate $S_1$. With $\mathscr{C}_1$ as defined in Section \ref{nonvanishing}, we can find a real
radial function $\psi\in C_0^{\infty}(\mathbb{R}^d)$ such
that supp$(\psi)\subset \mathscr{C}_1$, $0\leqslant \psi \leqslant 1$, and
\begin{equation}
\sum_{j=-\infty}^{\infty}\psi(\frac{y}{2^j})=1  \quad \textrm{for }y\in \mathbb{R}^d \setminus \{0\}.     \nonumber
\end{equation}
Denote
\begin{equation}
S_{1, M}=t^{\frac{d-1}{2}}\sum_{k\in \mathbb{Z}_{*}^{d}}\psi(M^{-1}k)|k|^{-\frac{d+1}{2}}K_{k}^{-\frac{1}{2}}\hat{\rho}(\varepsilon k)e(tH(k))\nonumber
\end{equation}
then $S_1=\sum_{j=0}^{\infty}S_{1, 2^j}$. It suffices to estimate $S_{1, M}$ for a fixed $M=2^j$ ($j\in \mathbb{N}_0$).

With the notation as in Section \ref{nonvanishing}, Lemma \ref{keylemma} ensures that there exists an allowable constant $A_3>0$ such that if $N\geqslant A_3$ is an integer then for every $\xi\in \mathscr{C}_1$ there exist linearly independent vectors $v_1(\xi), \ldots, v_d(\xi)\in \mathbb{Z}^d$ such that $|v_l|\asymp N$ ($1\leqslant l\leqslant d$), $|\textrm{det}(V)|\asymp N^d$, and
\begin{equation}
|h_k(y, v_1(\xi), \ldots, v_d(\xi))|\gtrsim N^{(k+2)d} \quad  \textrm{for $1\leqslant k\leqslant 3$, $y\in B(\xi, 2r)$},\nonumber
\end{equation}
where $r=1/2N$. The entries of $G_{k}(y, v_1(\xi), \ldots, v_d(\xi))$ satisfy the size estimate in Lemma \ref{keylemma}.

Since $\mathscr{C}_1$ is compact, we can find finitely many balls $\{B(\xi_i, r)\}_{i=1}^{I}$ ($\xi_i\in \mathscr{C}_1$ and $I\lesssim N^d$) and a partition of unity $\{\psi_i\}_{i=1}^{I}$ such that
\begin{enumerate}
\item  these balls have the bounded overlap property;
\item  $\mathscr{C}_1\subset \cup_{i=1}^{I}B(\xi_i, r)$;
\item  $\sum_{i}\psi_i(y)\equiv 1$ if $y\in \mathscr{C}_1$;
\item  $\psi_i\in C_0^{\infty}(B_i)$;
\item  $D^{\nu} \psi_i\lesssim N^{|\nu|}$.
\end{enumerate}
where we denote $B_i=B(\xi_i, r)$ and $B_i^*=B(\xi_i, 2r)$.

Denote
\begin{equation}
S_{1, M}^{(i)}=t^{\frac{d-1}{2}}\sum_{k\in \mathbb{Z}^d_*} U(k) e(tH(k)) \nonumber
\end{equation}
where
\begin{equation}
U(k)=\psi_i(M^{-1}k)\psi(M^{-1}k)|k|^{-\frac{d+1}{2}}K_{k}^{-\frac{1}{2}}\hat{\rho}(\varepsilon k), \nonumber
\end{equation}
then
\begin{displaymath}
S_{1, M}=\sum_{i=1}^{I}S_{1, M}^{(i)}.
\end{displaymath}
It suffices to estimate $S_{1, M}^{(i)}$ for a fixed $i$. Denote by $L$ the index of the lattice spanned by $v_1(\xi_i), \ldots, v_d(\xi_i)$ in the lattice $\mathbb{Z}^d$. Then $L=|\textrm{det}(V)|\asymp N^{d}$ and there exist vectors $b_l\in \mathbb{Z}^d$ ($1\leqslant l\leqslant L$) such that
\begin{equation}
\mathbb{Z}^d=\uplus_{l=1}^{L}(\mathbb{Z}v_1+\ldots+\mathbb{Z}v_d+b_l). \nonumber
\end{equation}
Let $N_1>0$ be an arbitrary integer $\geqslant \lceil \frac{d}{2} \rceil$. Applying above decomposition of $\mathbb{Z}^d$, for any $k\in \mathbb{Z}^d$ we can write $k=\sum_{s=1}^{d} m_s v_s+b_l$ where $m_s\in \mathbb{Z}$ ($1\leqslant s\leqslant d$). Hence
\begin{align}
S^{(i)}_{1, M} &=t^{\frac{d-1}{2}}\sum_{l=1}^{L}\sum_{m\in \mathbb{Z}^d}U(\sum_{s=1}^{d} v_s m_s+b_l) e(tH(\sum_{s=1}^{d} v_s m_s+b_l)) \nonumber\\
      &=t^{\frac{d-1}{2}}M^{-\frac{d+1}{2}}(1+|M\varepsilon|)^{-N_1}\sum_{l=1}^{L}S_l(T, \delta M; G, F),\nonumber
\end{align}
where $T=tM$, $\delta=N^{-1}$,
\begin{equation}
G(x)=M^{\frac{d+1}{2}}(1+|M\varepsilon|)^{N_1} U(M \sum_{s=1}^{d} \delta v_s x_s+b_l), \nonumber
\end{equation}
and
\begin{equation}
F(x)=H(\sum_{s=1}^{d} \delta v_s x_s+b_l/M). \nonumber
\end{equation}
We consider $F$ restricted to the convex domain
\begin{equation}
\Omega=\{x\in \mathbb{R}^d : \sum_{s=1}^{d} \delta v_s x_s+b_l/M \in B_i^* \}. \label{hh}
\end{equation}
If $\delta^{-1}<M$, $\Omega\subset c_0 B(0, 1)$ for an allowable constant $c_0$. The support of $G$ satisfies
\begin{equation}
\textrm{supp}(G)\subset \{x\in \mathbb{R}^d : \sum_{s=1}^{d} \delta v_s x_s+b_l/M \in \overline{B_i}\cap \mathscr{C}_1 \}\subset \Omega, \label{ii}
\end{equation}
and
\begin{equation}
\textrm{dist}(\textrm{supp}(G), \Omega^{c})\geqslant c'_1 \delta, \nonumber
\end{equation}
where $c'_1$ is an allowable constant. Note that
\begin{equation}
D^{\nu}U \lesssim \delta^{-|\nu|}M^{-\frac{d+1}{2}-|\nu|}(1+|M\varepsilon|)^{-N_1},  \nonumber
\end{equation}
and for all $x\in \Omega$, $1\leqslant i, j\leqslant d$, and $1\leqslant k\leqslant 3$
\begin{equation}
\frac{\partial^{k+2}F}{\partial x_1 \partial x_i \partial x_j \partial x_d^{k-1}}(x)=\delta^{k+2} g_{i, j}^{(k)}\big(\sum_{s=1}^{d} \delta v_s x_s+b_l/M, v_1(\xi_i), \ldots, v_d(\xi_i)\big),  \nonumber
\end{equation}
where $g_{i, j}^{(k)}$'s are as defined in Section \ref{nonvanishing}. It is not hard to check that assumptions of Theorem \ref{Exp-Sum} are satisfied.

If $d\geqslant 4$, we apply to $S_l(T, \delta M; G, F)$ Theorem \ref{Exp-Sum} with $q=1$, which determines the size of $\delta$, hence that of $N$. Note that $\delta$ is allowable, we will not write it explicitly in various bounds below. If $t\geqslant M\geqslant t^{1-\frac{2}{d}}$, the inequality $M>\delta^{-1}$ and restrictions of Theorem \ref{Exp-Sum} are both satisfied, thus
\begin{equation}
S_l(T, \delta M; G, F)\lesssim t^{\frac{d^2}{2(d^2+2d+4)}}M^{d-\frac{2d^2+d}{2(d^2+2d+4)}}, \nonumber
\end{equation}
which leads to
\begin{equation}
S_{1, M}=\sum S_{1, M}^{(i)}\lesssim t^{\frac{d-1}{2}+\frac{d^2}{2(d^2+2d+4)}}M^{\frac{d-1}{2}-\frac{2d^2+d}{2(d^2+2d+4)}} (1+|M\varepsilon|)^{-N_1}. \nonumber
\end{equation}

We split $S_1$ into three parts as follows:
\begin{equation}
S_1=\sum_{j=0}^{\infty}S_{1, 2^j}=\sum_{2^j< t^{1-\frac{2}{d}}}+\sum_{t^{1-\frac{2}{d}}\leqslant 2^j \leqslant t}+\sum_{2^j>t}S_{1, 2^j}.\nonumber
\end{equation}
The second sum is bounded by
\begin{equation}
\sum_{t^{1-\frac{2}{d}}\leqslant 2^j \leqslant t} t^{\frac{d-1}{2}+\frac{d^2}{2(d^2+2d+4)}}(2^j)^{\frac{d-1}{2}-\frac{2d^2+d}{2(d^2+2d+4)}} (1+|2^j \varepsilon|)^{-N_1}\nonumber
\end{equation}
\begin{equation}
   \lesssim t^{\frac{d-1}{2}+\frac{d^2}{2(d^2+2d+4)}}\varepsilon^{-\frac{d-1}{2}+\frac{2d^2+d}{2(d^2+2d+4)}}, \label{second-sum}
\end{equation}
while the first and third, by the trivial estimate, are bounded by $t^{d-2+1/d}$ and $1$ respectively. This finishes the estimate of $S_1$.

Note that the bound \eqref{err} for the Error term is smaller than \eqref{second-sum}, hence we get the bound for $R_{\varepsilon}(t)$. Since $t\pm C_1\varepsilon\asymp t$, we get the bound for $R_{\varepsilon}(t\pm C_1\varepsilon)$. Plugging these bounds in \eqref{basic} yields
\begin{equation}
P_{\mathcal{B}}(t)\lesssim  t^{d-2+\frac{1}{d}}+t^{\frac{d-1}{2}+\frac{d^2}{2(d^2+2d+4)}}\varepsilon^{-\frac{d-1}{2}+\frac{2d^2+d}{2(d^2+2d+4)}}
+t^{d-1}\varepsilon\nonumber
\end{equation}

Balancing the second and third terms yields
\begin{equation}
\varepsilon=t^{-\frac{d^3+2d-4}{d^3+d^2+5d+4}}.\nonumber
\end{equation}
With this choice of $\varepsilon$, the first term is smaller than the third one. Hence for $d\geqslant 4$
\begin{equation}
P_{\mathcal{B}}(t)\lesssim t^{d-2+\beta(d)}, \nonumber
\end{equation}
where $\beta(d)=(d^2+3d+8)/(d^3+d^2+5d+4)$.

If $d=3$, applying Theorem \ref{Exp-Sum} with $q=2$ yields $\beta(3)=73/158$. We omit the calculation since it is similar with above argument.
\end{proof}

{\it Remark:} To prove our exponent $\beta(d)$ for $d\geqslant 4$, we use the estimate of exponential sums obtained by using a ABAB-process (see Theorem \ref{Exp-Sum}). If we use more A- and B-processes we may further improve it at the cost of more technical difficulties. For example, an application of ABABAB-process may improve the exponent $\beta(d)$ by $1/d^3$.

\appendix
\section{Inverse function theorem}

We give a quantitative version of inverse function theorem below. We omit its proof since it is routine to prove it if we follow the proof in Rudin \cite{rudin}.

\begin{lemma}\label{app1}
Suppose $f$ is a $C^{(2)}$ mapping from an open set $\Omega\subset\mathbb{R}^d$
into $\mathbb{R}^d$ and $b=f(a)$ for some $a\in \Omega$. Assume $|\textrm{det} (Df(a))|$ $\geqslant$ $c$ and for any $x\in\Omega$
\begin{equation}
|D^{\alpha} f_i(x)|\leqslant C
\quad \quad \textrm{for $|\alpha|\leqslant 2$, $1\leqslant i\leqslant d$}.  \nonumber
\end{equation}
If $r_0\leqslant \sup\{r>0: B(a, r)\subset \Omega\}$, then $f$ is bijective from $B(a, r_1)$ to an open set containing $B(b, r_2)$ where
\begin{equation}
r_1=\min\{\frac{c}{2d^{7/2} (d-1)! C^d}, r_0\},  \nonumber
\end{equation}
\begin{equation}
r_2=\frac{c}{4d^{3/2}(d-1)!C^{d-1}}r_1. \nonumber
\end{equation}
The inverse function $f^{-1}$ is also a $C^{(2)}$ mapping.
\end{lemma}

{\it Remark:} Note that $r_2$ is linear in $r_1$. If $f$ is bijective from $B(a, r_1)$ to an open set containing $B(b, r_2)$, then for any $r'_1\leqslant r_1$, we can find the corresponding $r'_2$ such that $f$ is bijective from $B(a, r'_1)$ to an open set containing $B(b, r'_2)$.

\subsection*{Acknowledgments}
The subject of this paper was suggested by Professor Andreas Seeger. I would like to express my gratitude to him for his valuable advice and great help during the work.


\begin{thebibliography}{HD}




\normalsize
\baselineskip=17pt


\bibitem{Bentkus-Gotze}
V. Bentkus, F. G\"otze, \emph{On the lattice point problem for ellipsoids}, Acta. Arith. 80 (1997), 101--125.


\bibitem{Bonnesen}
T. Bonnesen, W. Fenchel, \emph{Theory of Convex Bodies}, BCS Associates, Moscow, Idaho, 1987.

\bibitem{bruna}
J. Bruna, A. Nagel, S. Wainger, \emph{Convex hypersurfaces and Fourier transforms},  Ann. of Math. 127 (1988), 333--365.


\bibitem{expsum}
S.W. Graham, G. Kolesnik, \emph{Van der Corput's Method
of Exponential Sums}, Cambridge Univ. Press, Cambridge, 1991.

\bibitem{heath-brown}
D.R. Heath-Brown, \emph{The growth rate of the Dedekind
Zeta-function on the critical line}, Acta. Arith. 49 (1988), 323-339.

\bibitem{Hlawka} E. Hlawka, \emph{\"Uber Integrale auf konvexen K\"orper I},
              Monatsh. Math. 54 (1950), 1-36.

\bibitem{hormander}
L. H\"ormander, \emph{The Analysis of Linear Partial
Differential Operators I}, Springer-Verlag, New
York, Berlin, 1983.

\bibitem{huxley}
M. N. Huxley, \emph{Area, Lattice Points, and Exponential Sums}, The Clarendon Press, Oxford Univ. Press, New York, 1996.

\bibitem{huxley2003} \bysame, \emph{Exponential sums and lattice points III}, Proc. London Math. Soc. 87 (2003), 591--609.

\bibitem{iosevich}
A. Iosevich, \emph{Lattice points and generalized Diophantine conditions}, J. Number Theory 90 (2001), 19-30.

\bibitem{I-S-S}
A. Iosevich, E. Sawyer, A. Seeger, \emph{Mean square discrepancy bounds for
the number of lattice points in large convex bodies},  J. Anal. Math. 87 (2002), 209--230.

\bibitem{kratzel}
E. Kr\"{a}tzel, \emph{Lattice points}, Kluwer Academic Publishers Group, Dordrecht, 1988.

\bibitem{nowak I}
E. Kr\"atzel, WG. Nowak, \emph{Lattice points in large
convex bodies}, Monatsh. Math. 112 (1991), 61-72.

\bibitem{nowak II}
\bysame, \emph{Lattice points in large convex bodies II},  Acta. Arith. 62 (1992), 285-295.


\bibitem{mullerI}
W. M\"uller, \emph{On the average order of the lattice rest of a convex body}, Acta. Arith. 80 (1997), 89-100.

\bibitem{mullerII}
\bysame, \emph{Lattice points in large convex bodies}, Monatsh. Math. 128 (1999), 315-330.

\bibitem{nowak-seq-1}
WG. Nowak, \emph{On the lattice rest of a convex body in $R^s$}, Arch. Math. 45 (1985), 284-288.

\bibitem{nowak-seq-2}
\bysame, \emph{On the lattice rest of a convex body in $R^s$ II}, Arch. Math. 47 (1986), 232-237.

\bibitem{nowak-seq-3}
\bysame, \emph{On the lattice rest of a convex body in $R^s$ III},  Czechoslovak Math. J. 41 (1991), 359-367.

\bibitem{rudin}
W. Rudin, \emph{Principles of Mathematical Analysis}, McGraw-Hill Book Company, Inc., New York, Toronto, London, 1953.


\bibitem{stein}
E.M. Stein, \emph{Harmonic Analysis: Real Variable Methods, Orthogonality, and Oscillatory Integrals}, Princeton Univ. Press, Princeton, 1993.

\bibitem{walfisz}
A. Walfisz, \emph{Gitterpunkte in mehrdimensionalen Kugeln}, Polska Akademia Nauk, Warszawa, 1957.

\end{thebibliography}
\end{document}